\documentclass[10pt]{amsart}
\usepackage{geometry}                
\newtheorem{theorem}{Theorem}

\newtheorem{conjecture}{Conjecture}
\newtheorem{proposition}{Proposition}

\usepackage{graphicx}
\usepackage{amssymb}
\usepackage{epstopdf}
\usepackage{parskip}
\usepackage{longtable}
\newcommand{\Mod}[1]{\ (\mathrm{mod}\ #1)}
\renewcommand{\Re}{\operatorname{Re}}
\renewcommand{\Im}{\operatorname{Im}}

\title{Shnirelman's Theorem Applications}
\author{Felix Sidokhine}
\date{\today}                                           

\begin{document}

\maketitle

\begin{abstract}
Shnirelman's theorem is applied to solving Diophantine equations, and also discussing of the problems of a representation of Gaussian integers  by a sum of odd Gaussian primes.
\end{abstract}

\section{Introduction}

In our note we study the systems of Diophantine equations and connect this problem to Shnirelman's theorem: ``Every integer $n > 7$ is a sum at most  $s_0$ (Shnirelman's constant) odd primes'' \cite{Ramare:2013aa}. The problem of the representation of Gaussian integers by a sum of odd Gaussian primes is explored from an angle of Shnirelman's theorem.

\section{Diophantine Equations and Shnirelman's Theorem}

 \begin{theorem}\label{th1}
 The system of diophantine equations:
 \begin{equation*}
 \begin{cases}
 x_{11} + x_{12} + x_{13} + x_{14} = a \\
 x_{21} + x_{22} + x_{23} + x_{24} = b \\
 \forall s \text{ } x_{1s} + x_{2s} = p_s \text{ where } p_s \text{ is an odd prime}
 \end{cases}
 \end{equation*}
 has a solution $(x_{11},...,x_{24})$  where $x_{ij} \in \mathbb{Z}_+$\footnote{$\mathbb{Z}_+ = \{ n \in \mathbb{Z} | n \geq 0\}$} if $a > 10$, $a \geq b > 0$ and $a + b \equiv 0 \Mod{2}$.
 \end{theorem}
 
 \begin{proof}
 The proof uses induction. Let the theorem be true up to some $n -2 = a + b$. Let the theorem be false for $n = a' + b'$. By Shnirelman's theorem with $s_0 = 4$, there exist $p \geq q \geq r \geq l \in \pi_\mathbb{Z}$\footnote{$\pi_\mathbb{Z} = \{ n \in \mathbb{Z} | n \text{ is prime} \}$} such that $n = a' + b' = p + q + r + l$.  
Let us consider 4 cases:
\begin{itemize}

\item Case 1: $b' \leq l$, then we have a solution: 
$\begin{cases}
 p + q + r + (l - b') = a' \\
 0 + 0 + 0 + b' = b' \\
\end{cases}$

\item Case 2: $l < b' \leq r + l$, then we have a solution: 
$\begin{cases}
p + q + (r + l - b') + 0 = a' \\
0 + 0 + (b' - l) + l = b'
\end{cases}$

\item Case 3: $r + l < b' \leq q +r + l$, then we have a solution: 
$\begin{cases}
p + (q + r + l - b') + 0 + 0 = a' \\
0 + (b'-r-l) +r + l = b'
\end{cases}$

\item Case 4: $q + r +l < b' < p + q+ r + l$, then we have a solution: 
$\begin{cases}
(p + q + r + l - b') + 0 + 0 + 0 = a' \\
(b' - q - r - l) + q + r + l = b'
\end{cases}$

From this, it follows that the theorem is true.
\end{itemize}
\end{proof}

 \begin{theorem}\label{th2}
 Given a system of diophantine equations where $a > 6$ and $a \geq b > 0$:
 \begin{equation*}
 \begin{cases}
 \sum_{i=1}^k x_{1i} = a \\
 \sum_{j=1}^k x_{2j} = b \\
 \forall s \text{ } x_{1s} + x_{2s} = p_s \text{ where } p_s \text{ is an odd prime}
 \end{cases}
 \end{equation*}
 there exists such $k$, where $k \leq k_0$ (constant), that the system has a solution $(x_{11},...,x_{2n})$  where $x_{ij} \in \mathbb{Z}_+$. 
 \end{theorem}
 
\begin{proof}
The proof mimics the one used for theorem \ref{th1}.
\end{proof} 
 
Based on these two theorems, we can propose the following conjecture which will be useful to use when working with Gaussian integers.

 \begin{conjecture}\label{cj1}
 Given the system of equations where $a > 6$ and $a \geq b > 0$:
 \begin{equation*}
 \begin{cases}
 \sum_{i=1}^k x_{1i} = a \\
 \sum_{j=1}^k x_{2j} = b \\
 \forall s \text{ } x_{1s}^2 + x_{2s}^2 = t_s \text{ where $t_s$ is either a prime } p_s \equiv 1 \Mod{4} \text{ or } t_s = p_s^2 \text{ where } p_s \equiv 3 \Mod{4}. 
 \end{cases}
 \end{equation*}
 there exists such $k$,  where $k \leq k_1$ (constant), that the system has a solution $(x_{11},...,x_{2n})$  where $x_{ij} \in \mathbb{Z}_+$. 
 \end{conjecture}

Let us consider $z \in \mathbb{Z}[i]$ and let $z = a + ib$. If the above conjecture is true, then it implies that $z$ is a sum of at most $k_1$ odd Gaussian primes with non-negative real and imaginary parts \cite{Knill:aa}.

\section{Gaussian Integers, Odd Gaussian Primes and Shnirelman's Theorem}

Let us introduce the notion of ``odd'' and ``even'' Gaussian integers, where $z$ is odd if $z \equiv 1 \Mod{1+i}$ and even if $z \equiv 0 \Mod{1+i}$. 

Our global experiment was to see whenever any element from the set $\Gamma_G = \{ z \in \mathbb{Z}[i] | \Re(z) > 0 \land -\Re(z) < \Im{z} \leq \Re(z) \}$ could be represented as a sum of at most $k_0$ elements from the set $\Gamma_\pi = \{ z \in \pi_{\mathbb{Z}[i]} | \Re(z) > 0 \land -\Re(z) < \Im(z) \leq \Re(z) \}$. Some of our computations and results are presented in appendix.

Our experimental work has allowed us to observe some interesting facts about Gaussian integers, notably given $z \in \Gamma_G$, and either $\Im(z) = \Re(z) - 1$ or $\Im(z) = \Re(z)$, $z$ cannot be represented as a sum of odd primes from $\Gamma_\pi$. This has led us to consider a different choice for our set $\Gamma_G = \{  z \in \mathbb{Z}[i] | \Re(z) > 0 \land \Im(z) \geq 0  \}$ and the following set $K_\pi = \{  z \in \pi_{\mathbb{Z}[i]} | \Re(z) \geq 0 \land \Im(z) \geq 0 \}$. It is worth noting that in both cases $\Gamma_G$  is a maximal set of non-associated Gaussian integers.
 
Studying the results from  \cite{Knill:aa} and our experimental data we have formulated the following conjecture:

\begin{conjecture}\label{cj2}
Given $A = \{ z \in \mathbb{Z}[i] | \Re(z) > 0 \land \Im(z) > 0 \}$ then for all $z$ in $A$, such that $\max(\Re(z),\Im(z)) \geq 7 $, $z$ is a sum of at most 3 odd  primes from the set $K_\pi$.
\end{conjecture}   
 
In our study of this conjecture, we have formulated the following proposition:

\begin{proposition}

Given that for any $z_0 \in A$ and some integer $k \geq 2$, such that $z_0 \equiv k \Mod{1+i}$ and $\max(\Re(z_0),\Im(z_0)) \geq c_0 > 0$, the equation $\sum_{i=1}^k x_i = z_0$ has a solution $(x_1,...,x_k)$ where $x_i \in K_\pi$. Then for any $w_0 \in A$ where $w_0 \equiv k+1 \Mod{1+i}$ and $\max(\Re(w_0),\Im(w_0)) \geq c_0 + 3$, the equation $\sum_{i=1}^{k+1} x_i = w_0$ has a solution $(x_1,...,x_{k+1})$ where $x_i \in K_\pi$

\end{proposition} 

\begin{proof}
The proof is by induction. Let the proposition be true for some $k_0$. Let it be false for $k_0 + 1$, therefore there exists such $w$ that $w \equiv k_0 + 1 \Mod{1+i}$ and $\max(\Re(w),\Im(w)) \geq c_0 + 3$ for which the proposition fails. Let $\Im(w) \geq c_0 + 3$, then $w - 3i \equiv k_0 \Mod{1+i}$ and $\max(\Re(w -3i),\Im(w -3i)) \geq c_0$ contradicting the inductive hypothesis. The case $\Re(w) \geq c_0 + 3$ is analogous.
\end{proof} 

If we assume that every odd Gaussian integer $z$ where $\Re(z)>0$, $\Im(z)>0$ and $\max(\Re(z),\Im(z)) \geq c_0$ is a sum of 3 odd Gaussian primes with non-negative real and imaginary parts, then we can use proposition 1 to get the following theorem:

\begin{theorem}
Any element $a \in \{ z \in \mathbb{Z}[i] | \Re(z) > 0 \land \Im(z) > 0 \land \max(\Re(z),\Im(z)) > c_1 \}$ is a sum of at most 4 odd Gaussian primes belonging to the set $K_\pi = \{ z \in \pi_{\mathbb{Z}[i]} | \Re(z) \geq 0 \land \Im(z) \geq 0 \}$
\end{theorem}

If we consider $\mathbb{Z}$ as a subset of $\mathbb{Z}[i]$, there are even elements of $\mathbb{Z}$ which can not be represented as a sum of 2  elements from the set $\pi_{\mathbb{Z}} \cap \pi_{\mathbb{Z}[i]}$. However, if one of the hypotheses below is true then the elements of $\mathbb{Z}$ as a subset of $\mathbb{Z}[i]$ could be represented as a sum of at most $s_0$ primes from $\pi_{\mathbb{Z}} \cap \pi_{\mathbb{Z}[i]}$.

\begin{itemize}
\item Hypothesis 1: There exists $c_0 > 0$ such that every integer $n \geq c_0$, $n \equiv 2 \Mod{4}$ is a sum of 2 primes where $p_i \equiv 3 \Mod{4}$
\item Hypothesis 2: There exists $c_0 > 0$ such that every integer $n \geq c_0$, $n \equiv 1 \Mod{4}$ is a sum of 3 primes where $p_i \equiv 3 \Mod{4}$
\item Hypothesis 3: There exists $c_0 > 0$ such that every integer $n \geq c_0$, $n \equiv 0 \Mod{4}$ is a sum of 4 primes where $p_i \equiv 3 \Mod{4}$
\item Hypothesis 4: There exists $c_0 > 0$ such that every integer $n \geq c_0$, $n \equiv 3 \Mod{4}$ is a sum of 5 primes where $p_i \equiv 3 \Mod{4}$
\end{itemize}

Indeed if at least one of the above hypotheses is true, then the following theorem is also true:

\begin{theorem}\label{th130}
For all $n \in \mathbb{Z}$ where $n \geq c_1$ and $c_1$ is a constant, there exists $m$ such that ${\sum_{i=1}^m x_i = n}$ has a solution $(p_1,...,p_m)$ where $\forall i$ $p_i \equiv 3 \mod 4$, and $p_i$ is prime and $m$ takes the values ranging from 2 to $s_1$ (constant). 
\end{theorem}

\begin{proof}
Let hypothesis 2 be true and consider the following 3 cases:
\begin{itemize}
\item $\forall n$ where $n \equiv 0 \Mod{4} \land n \geq c_0 + 3$ is a sum of four odd primes. Since $n -3 \equiv 1 \Mod{4}$ and $n -3 \geq c_0$  the proof follows hypothesis 2.
\item $\forall n$ where $n \equiv 3 \Mod{4} \land n \geq c_0 + 6$ is a sum of five odd primes. This is true due to the previous case.
\item $\forall n$ where $n \equiv 2 \Mod{4} \land n \geq c_0 + 9$ is a sum of six odd primes. This is true due to the previous case.
\end{itemize}
From these cases, we can conclude that $c_1 = c_0 + 9$ and $s_1 = 6$.
\end{proof}

Using conjecture \ref{cj2} and  theorem \ref{th130} we can formulate the following hypothesis:

\begin{conjecture}
For all $z \in \Gamma_G$, where $\max(\Re(z),\Im(z)) \geq c_0$, $z$ is a sum of at most $k_0$ odd primes from the set $K_\pi = \{ z \in \pi_{\mathbb{Z}[i]} | \Re(z) \geq 0 \land \Im(z) \geq 0\}$, where $k_0$ is a constant.

\end{conjecture}

Based on our experimental evidence, we can also propose the following conjecture:

\begin{conjecture}
For all $z \in \Gamma_G$, where $\max(\Re(z),\Im(z)) \geq 6$, $z$ is a sum of at most 3 odd Gaussian primes from the set $S_\pi = \{ z \in \pi_{\mathbb{Z}[i]} | \Re(z) \geq 0 \land -\Re(z) <  \Im(z)\}$ (i.e. $z = \sum_{l=1}^k p_l$ where $p_l \in S_\pi$ is an odd Gaussian prime and $||p_l|| < || z||)$.

\end{conjecture}

\section{Discussion and Conclusion}

In the present note we have constructed a solution of a system of Diophantine equations of the first degree in several variables over a set of non-negative integers by the Shnirelman theorem. We have demonstrated this approach on Gaussian integers. Based on experimental data on the representation of Gaussian integers by a sum of Gaussian primes we have shown that there is an infinite number of elements of a maximal set of non-associated Gaussian integers $\Gamma_G$ which cannot be represented as a sum of two or three Gaussian primes belonging to a set $K_\pi$ $(\Gamma_G \subset K_\pi \subset \mathbb{Z}[i])$. In this case we have formulated the Goldbach conjecture for $\mathbb{Z}[i]$ as some analogy of the Shnirelman theorem. However, according to our experimental evidence we expect that any Gaussian integer of $\Gamma_G$, excepting a finite number, can represent by a sum at most three odd Gaussian primes belonging to a set $S_\pi$ $(\Gamma_G \subset K_\pi \subset S_\pi \subset \mathbb{Z}[i])$ where $S_\pi$ is a minimal subset of $\mathbb{Z}[i]$ to satisfy such conditions.

\bibliography{references} 
\bibliographystyle{IEEEtran}

\newpage
\section{Appendix} 

\begin{longtable}{|l|l|l|l|l|}
\caption{Odd Gaussian integers as a linear combination of odd Gaussian primes where $\Gamma = \{ z \in \mathbb{Z}[i] | 0 \leq \Im(z) \leq \Re(z) \}$ and $\Gamma_\pi = \{ z \in \pi_{\mathbb{Z}[i]} |  \Re(z) > 0 \land -\Re(z) < \Im(z) \leq \Re(z) \}$ } \\
\hline
 $z \in \Gamma$ & \multicolumn{3}{l|}{$p,q,r  \in \Gamma_\pi; \max(||p||,||q||,||r||) < ||z||$} & Representation \\ \hline
$z=9$ & $p=3$ & $q=3$ & $r=3$ & $z=p+q+r$ \\ \hline
$z=13$ & $p=7$ & $q=3$ & $r=3$ & $z=p+q+r$ \\ \hline
$z=17$ & $p=11$ & $q=3$ & $r=3$ & $z=p+q+r$ \\ \hline
$z=21$ & $p=11$ & $q=7$ & $r=3$ & $z=p+q+r$ \\ \hline
$z=25$ & $p=11$ & $q=11$ & $r=3$ & $z=p+q+r$ \\ \hline
$z=29$ & $p=19$ & $q=7$ & $r=3$ & $z=p+q+r$ \\ \hline
$z=33$ & $p=19$ & $q=7$ & $r=7$ & $z=p+q+r$ \\ \hline
$z=37$ & $p=19$ & $q=11$ & $r=7$ & $z=p+q+r$ \\ \hline
\multicolumn{5}{}{} \\ \hline
$z=7$ & $p=3+2i$ & $q=2-i$ & $r=2-i$ & $z=p+q+r$ \\ \hline
$z=11$ & $p=7$ & $q=2+i$ & $r=2-i$ & $z=p+q+r$ \\ \hline
$z=15$ & $p=8+3i$ & $q=5-2i$ & $r=2-i$ & $z=p+q+r$ \\ \hline
$z=19$ & $p=11$ & $q=4+i$ & $r=4-i$ & $z=p+q+r$ \\ \hline
$z=23$ & $p=13+2i$ & $q=7$ & $r=3-2i$ & $z=p+q+r$ \\ \hline
$z=27$ & $p=13+2i$ & $q=11$ & $r=3-2i$ & $z=p+q+r$ \\ \hline
$z=31$ & $p=23$ & $q=4+i$ & $r=4-i$ & $z=p+q+r$ \\ \hline
$z=35$ & $p=23$ & $q=6+i$ & $r=6-i$ & $z=p+q+r$ \\ \hline
\multicolumn{5}{}{} \\ \hline
$z=6+i$ & $p=2+i$ & $q=2+i$ & $r=2-i$ & $z=p+q+r$ \\ \hline
$z=6+3i$ & $p=2+i$ & $q=2+i$ & $r=2+i$ & $z=p+q+r$ \\ \hline
$z=7+2i$ & $p=3$ & $q=2+i$ & $r=2+i$ & $z=p+q+r$ \\ \hline
$z=7+4i$ & $p=3+2i$ & $q=2+i$ & $r=2+i$ & $z=p+q+r$ \\ \hline
$z=8+i$ & $p=2+i$ & $q=3$ & $r=3$ & $z=p+q+r$ \\ \hline
$z=8+3i$ & $p=3+2i$ & $q=3$ & $r=2+i$ & $z=p+q+r$ \\ \hline
$z=8+5i$ & $p=3+2i$ & $q=3+2i$ & $r=2+i$ & $z=p+q+r$ \\ \hline
$z=9+2i$ & $p=3+2i$ & $q=3$ & $r=3$ & $z=p+q+r$ \\ \hline
$z=9+4i$ & $p=5+2i$ & $q=2+i$ & $r=2+i$ & $z=p+q+r$ \\ \hline
$z=9+6i$ & $p=5+4i$ & $q=2+i$ & $r=2+i$ & $z=p+q+r$ \\ \hline
$z=19+16i$ & $p=10+9i$ & $q=6+5i$ & $r=3+2i$ & $z=p+q+r$ \\ \hline
$z=27+24i$ & $p=23+22i$ & $q=2+i$ & $r=2+i$ & $z=p+q+r$ \\ \hline
$z=39+36i$ & $p=35+34i$ & $q=2+i$ & $r=2+i$ & $z=p+q+r$ \\ \hline
$z=48+45i$ & $p=25+24i$ & $q=20+19i$ & $r=3+2i$ & $z=p+q+r$ \\ \hline
$z=50+i$ & $p=23$ & $q=20+i$ & $r=7$ & $z=p+q+r$ \\ \hline
$z=50+3i$ & $p=40+i$ & $q=6+i$ & $r=4+i$ & $z=p+q+r$ \\ \hline
$z=50+7i$ & $p=40+i$ & $q=5+4i$ & $r=5+2i$ & $z=p+q+r$ \\ \hline
$z=50+11i$ & $p=36+5i$ & $q=9+4i$ & $r=5+2i$ & $z=p+q+r$ \\ \hline
$z=50+15i$ & $p=35+8i$ & $q=10+3i$ & $r=5+4i$ & $z=p+q+r$ \\ \hline
$z=50+19i$ & $p=35+16i$ & $q=10+i$ & $r=5+2i$ & $z=p+q+r$ \\ \hline
$z=50+25i$ & $p=35+18i$ & $q=10+3i$ & $r=5+4i$ & $z=p+q+r$ \\ \hline
$z=50+41i$ & $p=45+38i$ & $q=3+2i$ & $r=2+i$ & $z=p+q+r$ \\ \hline
$z=50+43i$ & $p=35+34i$ & $q=10+7i$ & $r=5+2i$ & $z=p+q+r$ \\ \hline
$z=50+45i$ & $p=25+24i$ & $q=20+19i$ & $r=5+2i$ & $z=p+q+r$ \\ \hline
$z=50+47i$ & $p=25+24i$ & $q=20+19i$ & $r=5+4i$ & $z=p+q+r$ \\ \hline
\multicolumn{5}{}{} \\ \hline
$z=6+5i$ & $p=3+2i$ & $q=2+i$ & $r=2-i$ & $z=p+q+ir$ \\ \hline
$z=7+6i$ & $p=3+2i$ & $q=3+2i$ & $r=2-i$ & $z=p+q+ir$ \\ \hline
$z=8+7i$ & $p=5+4i$ & $q=2+i$ & $r=2-i$ & $z=p+q+ir$ \\ \hline
$z=9+8i$ & $p=5+4i$ & $q=3+2i$ & $r=2-i$ & $z=p+q+ir$ \\ \hline
$x=11+10i$ & $p=8+3i$ & $q=2+i$ & $r=6-i$ & $z=p+q+ir$ \\ \hline
$z=19+18i$ & $p=10+9i$ & $q=8+5i$ & $r=4-i$ & $z=p+q+ir$ \\ \hline
$z=27+26i$ & $p=19+16i$ & $q=5+2i$ & $r=8-3i$ & $z=p+q+ir$ \\ \hline
$z=39+38i$ & $p=35+34i$ & $q=3+2i$ & $r=2-i$ & $z=p+q+ir$ \\ \hline
$z=48+47i$ & $p=23+22i$ & $q=20+19i$ & $r=6-5i$ & $z=p+q+ir$ \\ \hline
$z=49+48i$ & $p=25+24i$ & $q=20+19i$ & $r=5-4i$ & $z=p+q+ir$ \\ \hline
$z=50+49i$ & $p=25+24i$ & $q=20+19i$ & $r=6-5i$ & $z=p+q+ir$ \\ \hline
\end{longtable}

\begin{longtable}{|l|l|l|l|}
\caption{Even Gaussian integers as a linear combination of odd Gaussian primes where $\Gamma = \{ z \in \mathbb{Z}[i] | 0 \leq \Im(z) \leq \Re(z) \}$ and $\Gamma_\pi = \{ z \in \pi_{\mathbb{Z}[i]} |  \Re(z) > 0 \land -\Re(z) < \Im(z) \leq \Re(z) \}$ } \\
\hline
$z \in \Gamma$ & \multicolumn{2}{l|}{$p,q \in \Gamma_\pi; \max(||p||,||q||) < ||z||$} &Representation  \\ \hline
$z =6$ & $p=3$ & $q=3$ & $z=p+q$ \\ \hline
$z =10$ & $p=7$ & $q=3$ & $z=p+q$ \\ \hline
$z =14$ & $p=11$ & $q=3$ & $z=p+q$ \\ \hline
$z =18$ & $p=11$ & $q=7$ & $z=p+q$ \\ \hline
$z =30$ & $p=19$ & $q=11$ & $z=p+q$ \\ \hline
$z =42$ & $p=23$ & $q=19$ & $z=p+q$ \\ \hline
\multicolumn{4}{}{} \\ \hline
$z =  8$ & $p=5+2i$ & $q=3-2i$ & $z=p+q$ \\ \hline
$z =12$ & $p=6+i$ & $q=6-i$ & $z=p+q$ \\ \hline
$z =16$ & $p=13+2i$ & $q=3-2i$ & $z=p+q$ \\ \hline
$z =20$ & $p=10+i$ & $q=10-i$ & $z=p+q$ \\ \hline
$z =24$ & $p=12+7i$ & $q=12-7i$ & $z=p+q$ \\ \hline
$z =28$ & $p=15+2i$ & $q=13-2i$ & $z=p+q$ \\ \hline
$z =32$ & $p=17+2i$ & $q=15-2i$ & $z=p+q$ \\ \hline
$z =36$ & $p=21+4i$ & $q=15-4i$ & $z=p+q$ \\ \hline
$z =44$ & $p=27+2i$ & $q=17-2i$ & $z=p+q$ \\ \hline
\multicolumn{4}{}{} \\ \hline
$z =  6+2i$ & $p=3+2i$ & $q=3$ & $z=p+q$ \\ \hline
$z =  6+4i$ & $p=3+2i$ & $q=3+2i$ & $z=p+q$ \\ \hline
$z =  7+i$ & $p=4+i$ & $q=3$ & $z=p+q$ \\ \hline
$z =  7+3i$ & $p=5+2i$ & $q=2+i$ & $z=p+q$ \\ \hline
$z =  7+5i$ & $p=5+4i$ & $q=2+i$ & $z=p+q$ \\ \hline
$z =  9+7i$ & $p=6+5i$ & $q=3+2i$ & $z=p+q$ \\ \hline
$z= 19+17i$ & $p=13+12i$ & $q=6+5i$ & $z=p+q$ \\ \hline
$z= 27+25i$ & $p=25+24i$ & $q=2+i$ & $z=p+q$ \\ \hline
$z =36+2i$ & $p=26+i$ & $q=10+i$ & $z=p+q$ \\ \hline
$z= 39+37i$ & $p=36+35i$ & $q=3+2i$ & $z=p+q$ \\ \hline
$z= 48+46i$ & $p=25+24i$ & $q=23+22i$ & $z=p+q$ \\ \hline
$z= 49+47i$ & $p=26+25i$ & $q=23+22i$ & $z=p+q$ \\ \hline
$z= 50+48i$ & $p=25+24i$ & $q=25+24i$ & $z=p+q$ \\ \hline
$z =51+i$ & $p=47$ & $q=4+i$ & $z=p+q$ \\ \hline
$z =51+3i$ & $p=45+2i$ & $q=6+i$ & $z=p+q$ \\ \hline
$z =51+7i$ & $p=45+2i$ & $q=6+5i$ & $z=p+q$ \\ \hline
$z =51+9i$ & $p=47+8i$ & $q=4+i$ & $z=p+q$ \\ \hline
$z =51+25i$ & $p=48+23i$ & $q=3+2i$ & $z=p+q$ \\ \hline
$z= 51+31i$ & $p=49+30i$ & $q=2+i$ & $z=p+q$ \\ \hline
$z =51+49i$ & $p=48+47i$ & $q=3+2i$ & $z=p+q$ \\ \hline
\multicolumn{4}{}{} \\ \hline
$z =  6+6i$ & $p=5+4i$ & $q=2-i$ & $z=p+iq$ \\ \hline
$z =  7+7i$ & $p=5+4i$ & $q=3-2i$ & $z=p+iq$ \\ \hline
$z =  8+8i$ & $p= 7+2i$ & $q=6-i$ & $z=p+iq$ \\ \hline
$z =  9+9i$ & $p=5+4i$ & $q=5-4i$ & $z=p+iq$ \\ \hline
$z =10+10i$ & $p=6+5i$ & $q=5-4i$ & $z=p+iq$ \\ \hline
$z =11+11i$ & $p=6+5i$ & $q=6-5i$ & $z=p+iq$ \\ \hline
$z=12+12i$ & $p=11+6i$ & $q=6-i$ & $z=p+iq$ \\ \hline
$z=16+16i$ & $p=14+i$ & $q=15-2i$ & $z=p+iq$ \\ \hline
$z=19+19i$ & $p=18+17i$ & $q=2-i$ & $z=p+iq$ \\ \hline
$z=27+27i$ & $p=25+24i$ & $q=3-2i$ & $z=p+iq$ \\ \hline
$z=39+39i$ & $p=35+34i$ & $q=5-4i$ & $z=p+iq$ \\ \hline
$z=48+48i$ & $p=46+41i$ & $p=7-2i$ & $z=p+iq$ \\ \hline
$z=49+49i$ & $p=25+24i$ & $q=25-24i$ & $z=p+iq$ \\ \hline
$z=50+50i$ & $p=48+47i$ & $q=3-2i$ & $z=p+iq$ \\ \hline
$z=51+51i$ & $p=26+25i$ & $q=26-25i$ & $z=p+iq$ \\ \hline
\end{longtable}

\end{document}